\documentclass[a4paper,11pt]{article}
\usepackage{hyperref,makecell,amsmath,amssymb,amscd,amsthm,makeidx,txfonts,graphicx,url,bm,relsize}
\usepackage{a4wide,xy}
\usepackage{float,blkarray}
\usepackage{tikz}
\usepackage{amssymb}
\usepackage{tikz-cd}
\usepackage{geometry}
\let\oldmarginpar\marginpar
\renewcommand\marginpar[1]{\-\oldmarginpar[\raggedleft\footnotesize #1]
{\raggedright\footnotesize #1}}

\makeindex

\xyoption{all}
\input{xypic}

\numberwithin{equation}{section}

\newtheorem{thm}{Theorem}[section]
\newtheorem{proposition}[thm]{Proposition}
\newtheorem{corollary}[thm]{Corollary}
\newtheorem{conjecture}[thm]{Conjecture}

\newtheorem{lem}[thm]{Lemma}

\theoremstyle{remark}
\newtheorem{remark}[thm]{Remark}

\theoremstyle{definition}

\newcommand{\nam}[1]{\textcolor{teal}{#1}}

\def\SL{{\rm SL}}

\begin{document}

\title{Symplectic resolutions of moduli spaces of G-Higgs bundles}

\author{  GyeongHyeon Nam \\ {\it
   Aalto University} \\{\tt  gyeonghyeon.alg@gmail.com}  
 }
 \date{}

\pagestyle{myheadings}

\maketitle



\begin{abstract}The goal of this paper is to study the symplectic resolution of the moduli space of $G$-Higgs bundles over a compact Riemann surface with genus at least $2$.
\end{abstract}
\tableofcontents

\section{Introduction}

Let $G$ be a complex connected reductive group. Let us recall that a $G$-Higgs bundle on a compact Riemann surface $\Sigma$ is a pair $(P,\Phi)$, where $P$ is a principal holomorphic $G$-bundle on $\Sigma$ and $\Phi \in H^0(\Sigma,P(\mathfrak{g})\otimes K)$, where $K$ is the canonical bundle on $\Sigma$. Note that $\Phi$ is called the Higgs field. The notions of stable, semistable and polystable $G$-Higgs bundles are introduced in \cite[Definition 2.9]{GGM}. Note that this is an equivalent definition to semistable $G$-Higgs bundles in \cite{SimpsonII}.
Roughly speaking, a $G$-Higgs bundle $(P,\Phi)$ is polystable if it is semistable and the structure group of $P$ can be reduced to a   subgroup, and this gives a stable pair, and this corresponds to the process of looking at a polystable vector bundle as a direct sum of stable vector bundles of the same slope. 

In this paper, we consider 
the moduli space of  $G$-Higgs bundles with vanishing Chern classes  over a compact Riemann surface with genus $g\geq 2$   denoted  by $M_{Dol}^{vc}(G)$. It is well-known that there is a corresponding Betti moduli space (which is also called character variety) denoted by $M_B(G)$. When we study the symplectic geometry of $M_{Dol}^{vc}(G)$ and $M_B(G)$, we have the following natural question:
\begin{conjecture}\label{conj:main}
A moduli space $M_{Dol}^{vc}(G)$ admits   a symplectic resolution if and only if a moduli space $M_{B}(G)$ admits   a symplectic resolution.
\end{conjecture}
The parabolic version of this conjecture is also suggested in \cite[Conjecture 7.14]{ST}.
From this conjecture, we consider the $G$-Higgs bundle side corresponding to the following result:
\begin{thm}\cite[Theorem 7.1 and Corollary 7.3]{HSS24}Let us assume that $g>1$.
When $G$ is semisimple and its Dynkin diagram has no $A_1$-component or $g\geq 3$, then the irreducible component of   $M_B(G)$ containing the trivial representation does not admit a symplectic resolution.
\end{thm}
The reductive version of this result is discussed in \cite[\S7.4]{HSS24}.

\subsection{Main result}
It is known when the Betti moduli space over a compact Riemann surface admits a symplectic resolution or not in \cite{BS23, HSS24, Cheng}.  Using this result, we give the following evidence for Conjecture \ref{conj:main}.

  \begin{thm}\label{thm:g-geq2}(Theorem \ref{thm:higgs-sym-resol})
  Recall that $g\geq 2$.
Let us assume that $G$ is semisimple and  the Dynkin diagram of $G$ has no $A_1$-component or $g\geq 3$. Then the irreducible component of moduli space $M_{Dol}^{vc}(G)$ corresponding to the identity component of $M_B(G)$ does not admit a symplectic resolution.
 \end{thm}
 The  connected reductive group case is considered in Corollary \ref{coro:reductivecase}. 
Note that the cases of type $A$ groups  are  already
 well-studied in \cite{KY, TirelliHiggs}. 
\begin{thm}\cite{KY,TirelliHiggs}\label{thm:KYT}
The following hold true:
\begin{enumerate}
\item When $g \geq 1$ and $ n \geq 2$, the moduli space $M_{Dol}^{vc}(\mathrm{GL}_n)$ has symplectic singularities.
\item They admit projective symplectic resolutions exactly in the cases $g=1$ or $(g,n)=(2,2)$.
\end{enumerate}
\end{thm}

\begin{remark}It is natural to consider a symplectic resolution of $M_{Dol}^{vc}(G)$ when $g=1$. In fact, it is easy to check some cases, for example, when $G$ is $\mathrm{SO}_{2n+1}$ or $\mathrm{Sp}_{2n}$,   we can check that the identity component of $M_{Dol}^{vc}(G)$ admits a symplectic resolution (via the Hilbert scheme) using \cite[Theorem 4.34]{FGN}.
However, we would like to answer Conjecture \ref{conj:main} over every almost simple group and every component of $M_{Dol}^{vc}(G)$ (or $M_{Dol}(G)$), so we focus on $g \geq 2$ case in this paper.
\end{remark}

\subsection*{Acknowledgement}
The author would like to express profound gratitude to Gwyn Bellamy for his invaluable support and generous explanations throughout this work, and also thanks Mirko Mauri for his valuable comments. Furthermore, the author was supported by Oscar Kivinen's Väisälä project grant of the Finnish Academy of Science and Letters.

\section{Symplectic singularities}
\label{ss:proof12}
 
Let us prove the existence of symplectic singularities of $M_{Dol}^{vc}(G)$ using the formal neighbourhood of $M_B(G)$. Note that this result is proved in \cite[Proposition 6.28]{CFHM} for an arbitrary reductive group; their proof relies on a codimension estimate for $M_{Dol}^{vc}(G)$ with the work of Faltings \cite{Fal}.
In addition, this is also stated in \cite[Following paragraph of Theorem G]{HSS24} without proof. 

\begin{remark}
In fact, Proposition \ref{prop:codimension4} and Theorem \ref{thm:g>1,symsing} are well-known from \cite{HSS24,CFHM}.  
To be self-contained, we write a brief proof here again.
Note that the proof of the following proposition will be used to prove our main result.
\end{remark}

\begin{proposition}\label{prop:codimension4}Recall that $g\geq 2$. Let us assume that $G$ is semisimple, and either $G$ contains no simple factor of rank $1$ or $g \geq 3$.
The codimension of the singular locus of $M_{Dol}^{vc}(G)$ is at least $4$.
\end{proposition} 

\begin{proof} 
Let us recall that there are canonical isomorphisms between the formal neighbourhoods of $M_{Dol}^{vc}(G)$ and $M_B(G)$ from \cite[Theorem 10.6]{SimpsonII}. 
So we prove this proposition using $M_B(G)$. Recall that from the work of Simpson, the formal neighbourhood of a point $p$ of $M_B(G)$ is the formal completion of its tangent cone $C/\!\!/C_G(p)$, cf. \cite[Proposition 10.5]{SimpsonII}. Note that \cite[Theorem G (4)]{HSS24} implies that the codimension of the singular locus of the formal neighbourhood of $p$ at $M_B(G)$ is at least $4$ using \cite[\href{https://stacks.math.columbia.edu/tag/033A}{Tag 033A}]{stacks-project} and the fact that if the codimension of the singular locus of a variety $X$ is at least $m$ (i.e., Serre's condition $(R_{m-1})$), then its local rings also satisfy the condition $(R_{m-1})$.

Let us recall that   the codimension of the singular locus of the formal neighbourhood of $\tilde{p} \in M_{Dol}^{vc}(G)$ is equal to the codimension of the singular locus of the formal neighbourhood of the corresponding point of $M_B(G)$ (since the formal neighbourhood isomorphism preserves the singular locus). Then since the codimension of the singular locus of the formal neighbourhood of every point in $M_B(G)$ is at least $4$, so the codimension of the singular locus of the formal neighbourhood of every point in $M_{Dol}^{vc}(G)$ is at least $4$. This implies that the codimension of the singular locus of $M_{Dol}^{vc}(G)$ is at least $4$ from  \cite[\href{https://stacks.math.columbia.edu/tag/0353}{Tag 0353}]{stacks-project}  and \cite[\href{https://stacks.math.columbia.edu/tag/00MC}{Tag 00MC}]{stacks-project}.
\end{proof}

 \begin{thm}\label{thm:g>1,symsing}Recall that $g\geq 2$. Let us assume that  $G$ is semisimple,   and either $G$ contains no simple factor of rank $1$ or $g \geq 3$.
The moduli space $M_{Dol}^{vc}(G)$ has symplectic singularities.
 \end{thm}

 \begin{proof}
 Using the Isosingularity theorem (which shows that formal neighbourhoods of $M_{Dol}^{vc}(G)$ and $M_B(G)$ are isomorphic, cf. \cite[Theorem 10.6]{SimpsonII}), we can check that $M_{Dol}^{vc}(G)$ is normal with \cite[Theorem A]{Cheng}. (Note that the same method is used when $G=\mathrm{GL}_n$ in \cite[Theorem 4.6]{TirelliHiggs}.) Furthermore, this is quasi-projective, for example \cite{GO}.

To get the result, we show that $M_{Dol}^{vc}(G)$ has generically non-degenerate Poisson structure, and then by applying Flenner's theorem (since the codimension of the singular locus is at least $4$), we can finish the proof.
From \cite[Page 705]{GR} or \cite{BS}, the smooth locus of $M_{Dol}^{vc}(G)$ has a hyperk\"ahler structure, and so a holomorphic symplectic structure.  Then following the proof of \cite[Theorem 4.6]{TirelliHiggs} (since $M_{Dol}^{vc}(G)$ is equidimensional) and using Proposition \ref{prop:codimension4}, we can check that the smooth locus of $M_{Dol}^{vc}(G)$ admits a symplectic structure.\footnote{This is also studied in \cite[\S4]{BR94}.} 
Furthermore, we can check that $M_{Dol}^{vc}(G)$ has generically non-degenerate Poisson structure with the fact that the codimension of the singular locus is at least $4$ from Proposition \ref{prop:codimension4} again. So, by applying  Flenner's theorem, we can prove that $M_{Dol}^{vc}(G)$ has symplectic singularities.
 \end{proof}

 \section{Symplectic resolution}
 From the existence of symplectic singularities, the natural step is to consider its symplectic resolution.
Our goal is to prove the following.
 \begin{thm}\label{thm:higgs-sym-resol}Recall that $g\geq 2$. 
Let us assume that $G$ is semisimple and  the Dynkin diagram of $G$ has no $A_1$-component or $g\geq 3$. Then the irreducible component of moduli space $M_{Dol}^{vc}(G)$ corresponding to the identity component of $M_B(G)$ does not admit a symplectic resolution.
 \end{thm}

\begin{remark}
Recall that when $g=2$ and $G$ is a direct product of $\SL_2$'s, then the corresponding moduli space admits a symplectic resolution from \cite[Theorem 1.1]{KY}.
\end{remark}

 Before starting the proof of this result, we need to introduce the following lemma and a property of a cone.

 \begin{lem}\label{lem:singularopensubset}
Let us assume the same condition as in Theorem \ref{thm:higgs-sym-resol}.
The formal completion of the tangent cone of the trivial representation (denoted by $triv$) of $M_B(G)$ does not admit a symplectic resolution.
 \end{lem}

  \begin{proof}
Note that the formal neighbourhood of $triv$ is well-studied in \cite{HSS24}, and we use a proof of Proposition \ref{prop:codimension4} when $p=triv$ and so $C_G(p)=G$. Recall again  the formal neighbourhood of $triv$ of $M_B(G)$ is the formal completion of its tangent cone $C/\!\!/G$.  From \cite[\S7]{HSS24}, $C$ is $N_V$ for $V=g\mathfrak{g}$. 

  From a property of a conical symplectic variety, if the formal completion of $N_V/\!\!/G$ at the vertex admits a symplectic resolution, then $N_V/\!\!/G$ globally admits a symplectic resolution from the proof of \cite[Lemma 6.16]{BS21}. 
  Conversely, if $N_V/\!\!/G$ globally admits a symplectic resolution, then the formal completion of $C/\!\!/G$ at the vertex admits a symplectic resolution by restriction.
 Therefore, in the case of cone, $N_V/\!\!/G$ admits a symplectic resolution if and only if the formal completion of $N_V/\!\!/G$ at the vertex  admits a symplectic resolution. 
Therefore, our problem is reduced to check that $N_V/\!\!/G$ does not admit a symplectic resolution. 

Note that $N_V/\!\!/ G$ is singular from \cite[\S7.2]{HSS24} (if it is smooth, then obviously it admits a symplectic resolution). Let us prove this by showing that $N_V/\!\!/ G$ is factorial and has terminal singularities.
From,  \cite[Theorem 4.7]{HSS24} we directly get that  $N_V/\!\!/ G$ is factorial. 
We can check that $N_V/\!\!/ G$  has terminal singularities by showing that it has symplectic singularities and the codimension of its singular locus is at least $4$, cf. \cite[Corollary 1]{Nam}. We have that $N_V/\!\!/ G$ has symplectic singularities from \cite[Theorem F]{HSS24}. So if the codimension of the singular locus of $N_V/\!\!/ G$ is at least $4$, then we can get that $N_V/\!\!/ G$ has terminal singularities. We have that the codimension of the singular locus of the formal neighbourhood is at least $4$ from \cite[Theorem G (4)]{HSS24} and the proof of \ref{prop:codimension4}. 
Then \cite[\href{https://stacks.math.columbia.edu/tag/0353}{Tag 0353}]{stacks-project} gives that  $\mathcal{O}_{N_V/\!\!/G,triv}$ has $(R_3)$ property. 
With the same logic in \cite[Proof of Proposition 6.16]{BS21}, we have that $\mathbb{C}[N_V/\!\!/G]$ is an $\mathbb{N}$-graded, connected algebra. Then we have that $\mathbb{C}^\times$-action on $N_V/\!\!/G$ contracts to $triv$ (cf. \cite[\S2]{BPW}), and this gives that the $N_V/\!\!/G$ itself satisfy $(R_3)$ property.
Therefore, this gives that the codimension of the singular locus of $N_V/\!\!/ G$ is at least $4$. Recall that a singular factorial variety with terminal singularities does not admit a symplectic resolution, cf. \cite[Corollary 1.3]{Fu} or \cite[Proof of Theorem 6.13]{BS21}.
Therefore, we are done.
   \end{proof}

Now we are ready to prove Theorem \ref{thm:higgs-sym-resol}.
 
\begin{proof}[Proof of Theorem \ref{thm:higgs-sym-resol}]

For contradiction, let us assume that the corresponding component of  $M_{Dol}^{vc}(G)$ admits a symplectic resolution, then any formal neighbourhood of every point of the component of $M_{Dol}^{vc}(G)$ also admits a symplectic resolution, cf. \cite[\S4.1 (B)]{TirelliHiggs}.

Let $(Y,y)$ be the formal neighbourhood in $M_{Dol}^{vc}(G)$ which corresponds to the formal completion of $(TC_{triv}(M_B(G)) ,triv)$. 
Then since $(Y,y)$ admits a symplectic resolution,  this induces a symplectic resolution on the corresponding formal completion. 
However, we already know that the formal completion of $(TC_{triv}(M_B(G)) ,triv)$ does not admit a symplectic resolution from Lemma \ref{lem:singularopensubset}. Then this gives a contradiction, and therefore, we are done.
\end{proof}

\begin{remark}
Actually, our proof suggests a method to consider other closed orbit representations instead of the trivial representation in Lemma \ref{lem:singularopensubset}. Let us consider the tangent cone $C_p/\!\!/H$ at $p\in M_B(G)$, where $H=C_G(p)$. Then we need to check that when the tangent cone $C_p/\!\!/H$ does not admits a symplectic resolution. If this does not admits a symplectic resolution, then with the same login in the proof of Theorem  \ref{thm:higgs-sym-resol}, we can check that the corresponding component of $M_{Dol}^{vc}(G)$ does not admit a symplectic resolution.
\end{remark}

 \subsection{Reductive case}
 Let us assume that $G$ is a connected reductive group, and let $R(G):=\{(a_1,b_1,\ldots , a_g,b_g)\in G^{2g}\,|\, \prod_{i=1}^g[a_i,b_i]=1\}$. Then $R(G)\simeq (R(G_s)\times Z^{2g})/F^{2g}$, where $G=G_sZ$ (where $Z$ is the connected centre of $G$), and $F=G_s\cap Z$ from \cite[\S7.4]{HSS24}. 
 Then we can also consider its tangent cone similar to the semisimple group case since  $\mathfrak{g}=\mathfrak{g}_s\oplus \mathfrak{z(g)}$, where $\mathfrak{z(g)}$ is the Lie algebra of the centre of $G$.
Under the moment map, the central part $\mathfrak{z(g)}$ will be vanishing, so we can get the following result:
 \begin{corollary}\label{coro:reductivecase}Recall that $g\geq 2$.
Let us assume that $G$ is connected, reductive and  the Dynkin diagram of $G$ has no $A_1$-component or $g\geq 3$. If $G$ is not a torus, then the irreducible component of moduli space $M_{Dol}^{vc}(G)$ corresponding to the identity component of $M_B(G)$ does not admit a symplectic resolution.
 \end{corollary}
 \begin{proof}
 The logic of the proof is identical with the proof of Theorem \ref{thm:higgs-sym-resol}. Since   no non-trivial element of $F^{2g}$ 
 fixes  $triv$, \'etale-locally near $triv$ we have $R(G) \simeq R(G_s) \times Z^{2g}$. From this observation, we can see that the tangent cone of $M_B(G)$ at $triv$ is isomorphic to $  N_{V_s}/\!\!/G_s\times \mathbb{C}^{2g\dim(C)}$, where $V_s=g \mathfrak{g}_s$. Then we can prove the desired result with the same logic of Theorem \ref{thm:higgs-sym-resol} and Lemma \ref{lem:singularopensubset} (i.e., codimension, factorial and terminal singularities).
 \end{proof}

\end{document}